\documentclass[11pt,english]{amsart}

\usepackage{amsmath,amsfonts,amssymb,amsthm,amscd,textcomp,fontenc}
\usepackage{dsfont, mathrsfs,enumerate}
\usepackage{xspace}
\usepackage[T1]{fontenc}

\usepackage[english]{babel}

\usepackage{xy, xypic}
\usepackage{color} 
\usepackage{graphicx}
\usepackage{psfrag}  
\usepackage{pst-all}

\usepackage{stmaryrd} 

\DeclareMathOperator{\ad}{ad}

\renewcommand{\le}{\leqslant}
\renewcommand{\ge}{\geqslant}

\newcommand\Lie[1]{\mathfrak{#1}}
\newcommand{\g}{\Lie{g}}
\newcommand{\h}{\Lie{h}}

\newcommand{\n}{\Lie{n}}
\newcommand{\s}{\Lie{s}}
\def\sl{\mathfrak{sl}}

\newcommand{\m}{{\rm m}_{\rm Cl}}
\newcommand{\ev}{{\rm ev}}
\newcommand{\HC}{{\rm hc}}  % generalized HC 
\newcommand{\hc}{{\rm hc}_{\rm odd}}
\DeclareMathOperator{\Cl}{{\rm Cl}}
\newcommand{\F}{\mathcal{F}}
\newcommand{\J}{\mathcal{P}}

\def\c#1{\check{#1}}

\newcommand{\C}{\mathbb{C}}

\newcommand{\N}{\mathbb{N}}

\theoremstyle{plain}
\newtheorem{theorem}{Theorem}[section]

\newtheorem{proposition}[theorem]{Proposition}
\newtheorem{conjecture}[theorem]{Conjecture}

\theoremstyle{definition}

\newtheorem{remark}[theorem]{\em Remark}
\begin{document}

\title[On the Kostant conjecture for Clifford algebras]{On the Kostant conjecture for Clifford algebras}

\author{Anton Alekseev}
\address{Section de math\'ematiques, Universit\'e de Gen\`eve, 2-4 rue du Li\`evre,
c.p. 64, 1211 Gen\`eve 4, Switzerland}
\email{Anton.Alekseev@unige.ch}

\author{Anne Moreau}
\address{Laboratoire de Math\'ematiques et Applications, Universit\'e de Poitiers, 
T\'el\'eport 2, 11 Boulevard Marie et Pierre Curie, BP 30179, 
86962 Futuroscope Chasseneuil Cedex, France}
\email{Anne.Moreau@math.univ-poitiers.fr}

%\date{\today}

\begin{abstract}
Let $\g$ be a complex simple Lie algebra, and $\h \subset \g$ be a Cartan subalgebra.
In the end of 1990s, B. Kostant defined two filtrations on $\h$, one using the Clifford
algebras and the odd analogue of the Harish-Chandra projection $\hc: \Cl(\g) \to \Cl(\h)$,
and the other one using the canonical isomorphism $\check{\h} = \h^*$ (here 
$\check{\h}$ is the Cartan subalgebra in the simple Lie algebra $\check{\g}$
corresponding to the dual root system) and the adjoint action of the principal 
%${\rm sl}_2$
$\mathfrak{sl}_2$-triple $\check{\s} \subset \check{\g}$. Kostant conjectured that 
the two filtrations coincide.

The two filtrations arise in very different contexts, and comparing them proved to
be a difficult task. Y. Bazlov settled the conjecture for $\g$ of type $A_n$ using
explicit expressions for primitive invariants in the exterior algebra $\wedge \g$.
Up to now this approach did not lead to a proof for all simple Lie algebras.  

Recently, A. Joseph proved that the second Kostant filtration coincides with the filtration
on $\h$ induced by the generalized Harish-Chandra projection $(U\g \otimes \g)^\g \to S\h \otimes \h$
and the evaluation at $\rho \in \h^*$. 
In this note, we prove that Joseph's result is equivalent to the Kostant Conjecture.
We also show that the standard Harish-Chandra projection $U\g \to S\h$ composed
with evaluation at $\rho$ induces the same filtration on $\h$.

\end{abstract}


\maketitle

\section{Introduction}

Our starting point in this work is the fundamental paper by Kostant \cite{Kostant_Cliff}
on the structure of Clifford algebras over complex simple Lie algebras. 
Let $G$ be a complex simple Lie group of rank $r$ and $\g$ be the corresponding Lie algebra.
The cohomology ring of $G$ (over $\C$) is isomorphic to the ring of bi-invariant 
differential forms, $H(G) \cong (\wedge \g^*)^\g$.
By the Hopf-Koszul-Samelson theorem, $(\wedge \g^*)^\g$ is isomorphic to the 
exterior algebra $\wedge P$, where $P$ is a graded vector space with 
generators in degrees $2m_i+1$ defined by the exponents of $\g$, $m_1, \dots, m_r$.
For all $\g$ simple, $m_1=1$ and the corresponding bi-invariant differential
form is the Cartan 3-form
$$
\eta(x,y,z)=B_\g(x, [y,z]),
$$
where $x,y,z \in \g$ and $B_\g$ is the unique (up to multiple) invariant scalar
product on $\g$. 

The scalar product $B_\g$ defines a structure of a Clifford algebra
$\Cl(\g)=T\g/\langle x\otimes x - B_\g(x,x) \rangle$, where $T\g$ is the 
tensor algebra of $\g$. One of the main results of \cite{Kostant_Cliff} is the
theorem stating  that the adjoint invariant part of the Clifford algebra
$\Cl(\g)^\g \cong \Cl(P)$ is isomorphic to the Clifford algebra over $P$
(with the scalar product $B_P$ induced by $B_\g$). Under this isomorphism, the 
Cartan 3-form defines a canonical cubic element $\hat{\eta} \in \Cl(\g)$.
This element plays an important role in the theory of Kostant's cubic
Dirac operator (see \cite{Kostant_Dirac} and \cite{AM_Weil}), in the theory of group valued
moment maps \cite{AMW}, and recently in the Chern-Simons theory
in dimension one \cite{AM}.

Let $\g = \n_+ \oplus \h \oplus \n_-$ be a triangular decomposition
of $\g$, and let $\theta: \g \to \Cl(\g)$ be the canonical injection of
$\g$ as the generating set of $\Cl(\g)$. Consider the direct sum 
decomposition of the Clifford algebra
$$
\Cl(\g)=\Cl(\h) \oplus \big( \theta(\n_-) \Cl(\g) + \Cl(\g) \theta(\n_+)\big)
$$
and the {\em odd Harish-Chandra projection} $\hc: \Cl(\g) \to \Cl(\h)$.
The map $\hc$ is important in various applications ({\em e.g.} in the localization
formulas of \cite{AMW}).
In particular, the image of the canonical cubic element $\hat{\eta}$
is given by formula,
$$
\hc(\hat{\eta})=B_\g^\sharp(\rho),
$$
where $\rho \in \h^*$ is the half-sum of positive roots, and
$B_\g^\sharp: \h^* \to \h$ is the isomorphism induced by the scalar
product $B_\g$. In this context, the natural question is to evaluate the
images of higher degree generators of $P$ under the map $\hc$.
It is convenient to state this question in terms of the natural filtration of $P$
by degree (in contrast to grading, the filtration survives the passage to the 
Clifford algebra). A non-trivial result of Bazlov \cite{Bazlov}
and Kostant (private communications) is $\hc(P) =\theta(\h) \subset \Cl(\h)$.
Hence, $\hc(P^{(k)})$ defines a filtration on the Cartan
subalgebra $\h$.

Let $\check{\g}$ be the Lie algebra defined by the dual
root system, and $\check{\h} \subset \check{\g}$ the corresponding
Cartan subalgebra. Note that there is a canonical isomorphism 
$\h^* \cong \check{\h}$. The key observation of Kostant is that
$\rho \in \h^*$ viewed as an element of $\check{\h}$ coincides
with the Cartan element $\check{h}$ of the principal 
${\rm sl}_2$-triple $(\check{e}, \check{h}, \check{f})\subset \check{\g}$.
The action of this %${\rm sl}_2$
$\mathfrak{sl}_2$-triple induces a natural filtration
of $\check{\h}$,
$$
\F^{(m)}\check{\h}=\{ x  \in \check{\h}, \,\, \ad_{\check{e}}^{m+1} x=0\} .
$$
Since $[\check{e},[\check{e}, \check{h}]]=0$ , we have $\rho=\check{h}
\in \F^{(1)}\check{\h}$. Kostant suggested the following conjecture:

\vskip 0.2 cm

{\bf Kostant Conjecture.} The following two filtrations on $\h$ coincide:
$$
\hc(P^{(2m+1)}) = \theta\big( B_\g^\sharp(\F^{(m)}\check{\h}) \big) .
$$

\vskip 0.2cm

The two filtrations of the Kostant Conjecture arise in very different contexts,
and comparing them proved to be a difficult task. 
In \cite{Bazlov_thesis}, Bazlov
settled the Kostant conjecture for $\g$ of type $A_n$ using explicit expressions
for higher generators of $P$. 
The text \cite{Bazlov} claims to prove the conjecture
for all simple Lie algebras $\g$, but there is a gap in the argument (to be more
precise, in Lemma 4.4). 

The link between the two filtrations in the Kostant Conjecture can be established
in two steps. The decisive step has been recently made by Joseph in \cite{Joseph}.
He proved that filtration $B_\g^\sharp(\F^{(m)}\check{\h})$ arises in the
context of generalized Harish-Chandra projections recently introduced in \cite{KNV}.
In more detail, the generalized Harish-Chandra projection 
$\HC_\g: (U\g \otimes \g)^\g \to S\h \otimes \h$ composed with $\ev_\rho: S\h \to \C$, the evaluation at $\rho$, has the following property:

\vskip 0.2cm

{\bf Joseph's Theorem.} 

$$
(\ev_\rho \otimes 1) \circ \HC_\g (U^{(m)}\g \otimes \g)^\g = B_\g^\sharp(\F^{(m)}\check{\h}) .
$$

\vskip 0.2cm
Here $U^{(m)}\g$ is the natural filtration of the universal enveloping algebra 
by degree (where all elements of $\g \subset U\g$ are assigned degree one).

Our first result in this note is the equivalence of the Kostant Conjecture and Joseph's
theorem. In more detail, we prove that
$$
\theta\big( (\ev_\rho \otimes 1) \circ \HC_\g (U^{(m)}\g \otimes \g)^\g \big)
=\hc(P^{(2m+1)}).
$$
This results settles the Kostant Conjecture for all $\g$ simple.
Our proof is a combination of elementary Clifford calculus and of deep
results of \cite{Kostant_Cliff}.

Our second result is a variation of the same theme. For the standard Harish-Chandra
projection $\HC : U\g \to S\h$, we prove that
$$
\theta\big( (\ev_\rho \circ \HC \otimes 1) (U^{(m)}\g \otimes \g)^\g \big)= 
\hc(P^{(2m+1)}).
$$
While this paper was in preparation, A. Joseph informed us that he proved the following
theorem \cite{Joseph2}:
$$
(\ev_\rho \circ \HC \otimes 1) (U^{(m)}\g \otimes \g)^\g=
B_\g^\sharp(\F^{(m)}\check{\h})
$$
Together with our result, it gives an alternative proof of the Kostant Conjecture.

\vskip 0.2cm

{\bf Acknowledgments.} We are indebted to B. Kostant for introducing us to the
subject, and to E. Meinrenken for useful discussions and  encouragement. We 
would like to thank A. Joseph for sharing with us the results of his paper \cite{Joseph2}.
Our work was supported in part by the grants 200020-126817 and 
200020-126909  of the Swiss National Science
Foundation.

\section{Principal filtrations on $\h$}

Let $\g$ be a complex finite dimensional simple Lie algebra 
of rank $r$ and 
\begin{eqnarray*}   
	\g = \n_- \oplus \h \oplus \n_+ 
\end{eqnarray*}
a triangular decomposition of $\g$. 
Let $\Delta \subset \h^*$ be the root system of $(\g,\h)$, 
$\Pi$ the system of simple roots 
with respect to the Borel subalgebra $\h \oplus \n_+$ 
and $\Delta^+$ the corresponding set of positive roots. 
For any $\alpha \in \Delta$, we denote by $\c{\alpha} \in \h$ its coroot. 
Choose for each $\alpha \in \Delta$ a nonzero element $e_\alpha$ in 
the $\alpha$-root space so that 
$\{ \c{\beta}, e_{\alpha} \ ; \  \beta \in \Pi, \, \alpha \in \Delta \}$ 
is a Chevalley basis of $\g$. 

\medskip

Let $e=\sum_{\beta \in \Pi} e_\beta \in \n_+$ be the principal nilpotent element. 
%(here $e_i$'s are Chevalley generators of $\n_+$). 
By Jocobson-Morosov theorem, one can form an $\sl_2$-triple $(e,h,f)$ with
$h \in \h$ and $f \in \n_-$. 
The triple $(e,h,f)$ spans the principal $\sl_2$ Lie subalgebra 
$\s \subset \g$. 
Using the adjoint action of $\s$ on $\g$, one can define a filtration of $\h$,
\[
	\F^{(m)}\h = \{ x \in \h, \,\, \ad_e^{m+1} x=0 \} .
\]
The dimension of the vector space $\F^{(m)}\h$ jumps at the values $m=m_1, \dots, m_r$,
the exponents of $\g$. 
This follows from the Kostant's theorem \cite{Kostant_sl2}
which shows that $\g$ is a direct sum of $r$ irreducible $\ad_\s$-modules
of dimensions $2m_i+1$. 
In most cases, the exponents $m_1, \dots, m_r$
are all distinct. 
The exception is the case of the $D_n$ series, for even $n$ and $n \ge 4$, when 
there are two coincident exponents (equal to $n-1$).  

The filtration $\F^{(m)}\h$ is induced by a grading. 
In more detail, 
let $B_\g$ be a unique (up to multiple) invariant scalar product on $\g$.
It restricts to a scalar product on $\h$. 
The orthogonal complements
$\F^{(m)}\h^\perp$ form a new filtration. 
Define
$\F_m\h=\F^{(m)}\h \cap \F^{(m-1)}\h^\perp$. 
Note that
$\F_m\h$ are non-empty only for $m=m_i$ 
(and in most cases these are complex lines). 
Then, $\F^{(m)}\h = \oplus_{k\le m} \F_k\h$.

\medskip

Let $\c{\g}$ be %the Langlands dual of $\g$, that is 
the simple Lie algebra corresponding 
to the dual root system $\c{\Delta}$ 
of $\Delta$. 
The Cartan subalgebra $\c{\h}$ (with respect to $\c{\Delta}$) 
of $\c{\g}$ is canonically isomorphic to $\h^*$. 
As a consequence, we can identify $\c{\h}$ 
to $\h^*$.  
Let $\c{e} =\sum_{\beta \in \Pi} \c{e}_{\c{\beta}}$ be the principal nilpotent
element of $\c{\g}$
(here $\{ \beta, \, \c{e}_{\c{\alpha}} \ ; \  \beta \in \Pi, \, \alpha \in \Delta \}$ 
is a Chevalley basis of $\c{\g}$).
%Note that there are canonical isomorphism $\c{\h} = \h^*$
%and $\c{\h}^* = \h$. 
Since $\c{\h}^*$ and $\h$ are canonically  isomorphic,  
the co-adjoint action of $\c{e}$
induces another filtration on $\h$,
\[�
	\c{\F}^{(m)}\h = \{ x \in \h, \,\,  \big( \ad^*_{\c{e}} \big)^{m+1} x=0 \} .
\] 
The restriction of $B_{\c{\g}}$ to $\c{\h}$ is nondegenerate, 
hence it induces an isomorphism, $B_{\c{\g}}^\flat$, between $\c{\h}$ and $\c{\h}^*$ 
and a nondegenerate bilinear form on $\c{\h}^* \cong \h$.  
Observe that $\c{\F}^{(m)}\h = B_{\c{\g}}^\flat ( \c{\F}^{(m)} \c{\h} )$  
where $\c{\F}^{(m)} \c{\h} = \{ x \in \c{\h}, \, \ad_{\c{e}}^{m+1} x=0 \}$. 
% 
%The exponents are completely determined by the Weyl group. 
%Hence, $\g$ and $\c{\g}$ share the same set of exponents $\{m_1, \dots, m_r\}$. 
Similarly to $\F^{(m)}\h$, the filtration $\c{\F}^{(m)}\h$ 
is induced by a grading, $\c{\F}^{(m)}\h=\oplus_{k\le m} \c{\F}_k\h$, 
with $\c{\F}_m\h=\c{\F}^{(m)}\h \cap \c{\F}^{(m-1)}\h^\perp$. 

\begin{remark}

The bilinear forms $B_{\c{\g}}|_{\h \times \h}$ 
and $B_{\g}|_{\h \times \h}$ do not coincide in general 
but they are proportional to each other. 
Hence, the orthogonal complements are defined relative to $B_{\c{\g}}|_{\h \times \h}$ 
or, equivalently, to $B_{\g}|_{\h \times \h}$.

\end{remark}

%%%%%%%%%%%%%%%%%%
%%%
\section{Clifford algebra and odd Harish-Chandra projection}
%%%
%%%%%%%%%%%%%%%%%%

Let $\Cl(\g)=\Cl(\g, B_\g)$ be the Clifford algebra defined by the scalar product $B_\g$.  
Recall that $\Cl(\g)$ is the quotient of the tensor algebra $T\g$ 
by the ideal generated by elements $x\otimes x - B_\g(x,x)$, $x \in \g$. 
Denote by $xy$ the product in $\Cl(\g)$, for $x,y \in \Cl(\g)$, 
and by $\theta : \g \to \Cl(\g)$ the canonical injection of $\g$ 
as a generating subspace of $\Cl(\g)$. 

\smallskip

As a vector space,
$\Cl(\g)$ is isomorphic to the exterior algebra $\wedge \g$, with an explicit
$\ad_\g$-equivariant isomorphism given by the anti-symmetrization
map 
\[� 
   q : \wedge \g \to \Cl(\g) .
\] 
Note that $q$ maps the
natural filtration on $\wedge \g$ (with all $x \in \g$ of degree one)
to the natural filtration on $\Cl(\g)$.
Set $\Cl^{0}(\g) = q \big( \sum_{k}  \wedge^{2k} \g \big)$ 
and $\Cl^{1}(\g) = q \big( \sum_{k} \wedge^{2k+1} \g \big)$. 
Then $\Cl^{i}(\g) \Cl^{j}(\g) \subset \Cl^{i+j}(\g)$ 
for $i,j \in \{0,1\}$, where we set $\Cl^2(\g) = \Cl^0(\g)$.  
Hence, $\Cl(\g)$ inherits the structure of a super 
associative algebra and so has the structure of a super Lie algebra, $[\,\, , \,]$,  
defined by: 
\[
	[u,v] = uv - (-1)^{ij} vu, \qquad u \in \Cl^{i}(\g), \, v \in \Cl^{j}(\g) .   �
\] 
Note that for $x,y \in \g$, we have 
$[\theta(x),\theta(y)] = \theta(x)\theta(y) + \theta(x)\theta(y) = 2 B_\g(x,y)$. 
To avoid confusion, we will denote by $[\,\, ,\, ]_\g$ 
the usual Lie bracket on $\g$. 

%Let $\Delta: \wedge \g \to \wedge \g \otimes \wedge \g$ be the co-product defined
%by the formula 
%$\Delta(x)=x\otimes 1 + 1 \otimes x$ for $x\in \g$. 
%Namely, 
%$\Delta (x_1 \wedge \cdots \wedge x_k) = \sum_{p=0}^k \sum_{\sigma \in \mathfrak{S}_{p,k-p}} 
%\mathrm{sgn}(\sigma) x_{\sigma 1} \wedge \cdots \wedge x_{\sigma p} 
%\otimes x_{\sigma (p+1)} \wedge \cdots \wedge x_{\sigma k}
%$. 
%Denote by $\Delta_0: (\wedge \g)^\g \to (\wedge \g)^\g \otimes (\wedge \g)^\g$ 
%the map induced from $\Delta$ by the 
%direct sum decomposition $\wedge \g = (\wedge \g)^\g \oplus \ad_\g (\wedge \g)$ 
%and let 
%\[�
%P := \{ v \in (\wedge \g)^\g, \ \Delta_0 (v)=v \otimes 1 + 1 \otimes v \}
%\] 
%be the space of primitive elements. 

\smallskip

The scalar product $B_\g$ extends to a scalar product on $\wedge \g$, again 
denoted by $B_\g$, in the standard way. 
Let $P \subset (\wedge \g)^\g$ be the set of primitive element 
which is the $B_\g$-orthocomplement to 
$(\wedge \g)^\g_+ \wedge (\wedge \g)^\g_+$ in $(\wedge \g)^\g_+$ 
where $(\wedge \g)^\g_+$ is the augmentation ideal of $(\wedge \g)^\g$. 
Let $\alpha$ be the unique anti-automorphism of $\wedge \g$ 
which is equal to identity on $\g$,  
and let $B_{P}$ be the bilinear form 
defined on $P\times P$ by $B_{P} (u,v) = B_\g (\alpha u, v)$.  

The set $P$ is a graded vector space of dimension $r$ 
with generators in degrees $2m_i+1$, $i=1,\dots,r$, 
see \cite[\S 4.3]{Kostant_Cliff}. 
%and Theorem B in \cite{Kostant_Cliff} states that
%$\Cl(\g)^\g \cong \Cl(P, B_{P})$ with isomorphism 
%induced by the injection $q: P \to \Cl(\g)^\g$. 
%
We will be using the natural filtration
\[�
	P^{(k)} = \{ v \in P, \,\,  \deg (v) \le k \}, \quad k \in \N. 
\]
Again, one can introduce the grading
with $P^{(k)}=\oplus_{l\le k} P_l$ and $P_k=P^{(k)} \cap (P^{(k-1)})^\perp$.
The graded components $P_k$ are nonvanishing for $k = 2 m_i+1$, $i=1,\ldots,r$.
Set $\J = q (P)$ and define the filtration
\[�
	\J^{(k)} = \{ v \in \J, \,\, {\rm deg}(v) \le k \}, \quad k \in \N. 
\] 
Obviously, $q(P^{(k)})=\J^{(k)}$. 

\medskip

One can easily show that there is the following decomposition
\[�
  \Cl(\g) = \Cl(\h) \oplus \big( \theta(\n_-) \Cl(\g) + \Cl(\g) \theta(\n_+) \big),
\]
where $\Cl(\h) \subset \Cl(\g)$ is the subalgebra spanned by $\theta(\h)$.
The projection $\hc : \Cl(\g) \to \Cl(\h)$ with respect to this decomposition is called
the odd Harish-Chandra projection.  
When restricted to $\ad_\g$-invariant elements, 
the map $\hc \circ q$ is an isomorphism, see \cite[Theorem 4.1]{Bazlov}. 
Bazlov (\cite[Proposition 4.5]{Bazlov}) and Kostant (private communications) 
proved the 
following non-trivial property of the map $\hc$: 

\begin{theorem}[Bazlov-Kostant]  \label{BK}

We have: $\hc \circ q(P) = \theta(\h).$

\end{theorem}

\medskip 

The following conjecture is due to Kostant (see also \cite[\S 5.6]{Bazlov}): 

\begin{conjecture}[Kostant]    \label{Kostant_conj}

For any $k \in \N$, we have: $\hc(\J^{(2k+1)}) = \theta\big( \c{\F}^{(k)} \h \big).$

\end{conjecture}

\begin{remark}

In fact, the original formulation of Kostant's conjecture is more precise. 
It states that $\hc \big(q(P_{2k+1}) \big)= \theta \big(\c{\F}_k\h\big)$.
However, it is sufficient to prove the statement about filtrations 
as we explain below.

\smallskip

The map $\hc\circ q: P \to \theta(\h)$ is an isomorphism of vector spaces, and 
the scalar product $B_{P}$ is mapped to $B_\g$ in the following sense:  
Let $v_1,v_2 \in P$ and $x_1,x_2 \in \h$ 
such that $\hc(q(v_1)) =\theta(x_1)$ and $\hc(q(v_2)) =\theta(x_2)$ (see Theorem \ref{BK}).  
Theorem~B in \cite{Kostant_Cliff} states that
$\Cl(\g)^\g \cong \Cl(P, B_{P})$ with isomorphism 
induced by the injection $q: P \to \Cl(\g)^\g$. 
As a consequence, 
we have $ [q(v_1), q(v_1)]  = 2 B_{P} (v_1 , v_2)$. 
Since in addition the restriction of $\hc \circ q$ to $\ad_\g$-invariant elements is an isomorphism, 
we get
\begin{eqnarray*}
	 2 B( x_1, x_2) & = & [ \theta(x_1) , \theta(x_2)] =
 	[ \hc(q(v_1)) , \hc(q(v_2))] \\
	& = & \hc \big( [q(v_1), q(v_2)] \big) \,\, = \,\, 2 B_{P} (v_1,v_2).
\end{eqnarray*}

Hence, 
$\hc(q(P^{(2k+1)})) = \theta(\c{\F}^{(k)}\h)$ implies 
$\hc(q(P^{(2k+1)})^\perp) = \theta(\c{\F}^{(k)}\h^\perp)$.
Note that $P^{(2k)}=P^{(2k-1)}$. 
Indeed, remember that the graded components $P_k$ are nonvanishing 
for $k = 2 m_i +1$, $i=1,\ldots,r$. 
So, 
\begin{eqnarray*}
	\hc(q( P_{2k+1} )) & = & \hc(q( P^{(2k+1)} \cap (P^{(2k)})^\perp )) \\
	&  = &  \hc(q( P^{(2k+1)} \cap (P^{(2k-1)})^\perp ))  
	     =  \theta( \c{\F}^{(k)}\h \cap \c{\F}^{(k-1)}\h^\perp ) \,\, = \,\, \theta(\c{\F}_k\h).
\end{eqnarray*}

\end{remark}

%%%%%%%%%%%%%
%%%
\section{Joseph's theorem implies the Kostant conjecture} 
%%%
%%%%%%%%%%%%%

Recall that $U\g$ admits a decomposition, 
\[
	U\g = S\h \oplus \big( \n_- \, U\g + U\g \, \n_+).
\]
It induces the Harish-Chandra projection $\HC: U\g \to S\h$.
By the Harish-Chandra theorem, the center of the universal
enveloping algebra is isomorphic (under the map $\HC$) to the 
ring $(S\h)^{\cdot W}$ where the Weyl group $W$ acts on
$\h^*$ by the shifted action, 
\[
\lambda \mapsto w\cdot \lambda = (\lambda+\rho)^w - \rho, \qquad w \in W,
\]
where $\lambda \mapsto \lambda^w$ stands for the usual $W$-action.
Recently, generalized Harish-Chandra projections were introduced
in \cite{KNV}. 
Let $V$ be a finite-dimensional 
%irreducible \qA{:check that this is necessary: As it seems no need to make this assumption 
% (see the beginning of \cite[\S 4]{KNV}.} 
$\g$-module, and
denote by $\sigma$ the algebra homomorphism $\sigma: U\g \to {\rm End}(V)$. 
The space $U\g \otimes V$ carries two commuting $\g$-actions,
\[�
	\rho_L(x): a \otimes b \to xa \otimes b \,  , \qquad 
	\rho_R(x): a\otimes b \to - ax \otimes b + a \otimes \sigma(x)b.
\]
Using these two actions, one defines a direct sum decomposition,
\[�
	U\g \otimes V = S\h \otimes V \oplus \big( \rho_L(\n_-) (U\g \otimes V)
	+ \rho_R(\n_+) (U\g \otimes V) \big), 
\]
see the equality (4) in \cite{KNV} which refers to \cite[Proposition 3.3]{KO}. 
This decomposition defines the generalized Harish-Chandra projection, 
\[
	\HC_V :  U\g \otimes V \to S\h \otimes V.
\]
For $V$ the trivial module, $\HC_V$ coincides with the standard 
Harish-Chandra projection $\HC$. 

Under the generalized Harish-Chandra projection, the invariant subspace
$(U\g \otimes V)^\g$ for the diagonal action $\rho(x)=\rho_L(x) + \rho_R(x)$
injects into $S\h \otimes V[0]$ (here $V[0] \subset V$ is the zero weight
subspace relative to the action of $\h$); 
see for instance \cite[\S 3]{KNV}. 
In particular, for $V=\g$ equipped with the adjoint action, one
gets an injection
\[�
	\HC_\g: (U\g \otimes \g)^\g \hookrightarrow S\h \otimes \h.
\]

For every element $\lambda \in \h^*$, one can introduce an evaluation
map, $\ev_\lambda: S\h \to \C$, associating to a polynomial $p\in S\h \cong \C[\h^*]$
its value at $\lambda$, $\ev_\lambda(p) = p(\lambda)$. 
In particular, we will be interested in the 
evaluation at $\rho$, the half-sum of positive roots. 
Recently, Joseph \cite{Joseph} proved the following theorem:

\begin{theorem}[Joseph]    \label{Joseph}

For any $m \in \N$, we have:  
\[�
	(\ev_\rho \otimes 1) \circ \HC_\g \big( (U^{(m)}\g \otimes \g)^\g \big) 
	= \c{\F}^{(2m+1)}\h.
\]

\end{theorem}

Here $U^{(m)}\g$ is the natural filtration of $U\g$ induced by assigning
degree one to all elements of $\g \subset U\g$. 
The proof of Joseph's theorem involves the technique of 
Zhelobenko and Bernshtein-Gelfand-Gelfand operators. 

\medskip

Our first result in this note is the following theorem:

\begin{theorem} \label{Kostant}  

Joseph's Theorem is equivalent to the Kostant conjecture. 

\end{theorem}

The rest of this section is devoted to the proof of 
Theorem \ref{Kostant}. 
To advance, we formulate several key propositions.
Let 
\[�
	\tau: U\g \to \Cl(\g)
\]
be the unique algebra homomorphism defined
by the properties $[ \tau(x) , \theta(y) ]=\theta( [x,y]_\g )$ 
and ${\rm deg}(\tau(x))=2$, for $x,y \in \g$. 
Let 
\[�
	\m : \Cl(\g) \otimes \g \to \Cl(\g), \ a \otimes b \to a\theta(b)
\] 
be the product map
in the Clifford algebra, and
\[�
	\mu = \m \circ (\tau \otimes 1): U\g \otimes \g \to \Cl(\g).
\]
Since $\tau$ has degree two, note  
that $\mu$ maps $U^{(m)}\g \otimes \g$ to $\Cl^{(2m+1)}(\g)$. 
Our first proposition is the following: 

\begin{proposition}    \label{prop1}

The subspaces $\mu\big(\rho_L(\n_-)(U\g \otimes \g)\big)$ and 
$\mu\big(\rho_R(\n_+)(U\g \otimes \g)\big)$
are contained in the kernel of the odd Harish-Chandra projection $\hc$.

\end{proposition}

\begin{proof}

To start with, observe that $\tau(\n_-) \subset \theta(\n_-)\Cl(\g)$ 
and $\tau(\n_+) \subset \Cl(\g) \theta(\n_+)$.  
Indeed, elements of $\tau(\n_-)$ (resp.\ $\tau(\n_+)$) 
are of negative (resp.\ positive) weight under the adjoint $\h$-action.

For the first subspace, we have
\[
	\mu \big(  \rho_L(\n_-) (U\g \otimes \g) \big)
	\subset \m \big( \tau(\n_-) \tau(U\g) \otimes \g \big)
	\subset \tau(\n_-) \Cl(\g) \subset \theta(\n_-) \Cl(\g). 
\]
For the second subspace, a more detailed analysis is needed: 
For $a \in U\g$ and $x,b \in \g$, one has, 
\begin{eqnarray*}
	\mu \big( \rho_R(x) (a \otimes b)\big)  
		& = & \m \circ (\tau \otimes 1) \big( - ax \otimes b + a \otimes [x,b]_\g  \big)\\
		& = &  - \tau(a) \tau(x) \theta(b) + \tau(a) \theta([x,b]_\g) 
		\,\,  = \,\,� - \tau(a) \theta(b) \tau(x).
\end{eqnarray*}
Hence, 
\[
	\mu\big( \rho_R(\n_+) (U\g \otimes \g) \big)
	\subset \Cl(\g) \tau(\n_+) \subset \Cl(\g) \theta(\n_+).
\] 

By definition, both $\theta(\n_-) \Cl(\g)$ and $\Cl(\g) \theta(\n_+)$ are contained
in the kernel of $\hc$ and the proposition follows.

\end{proof}

The next fact is proved for instance in Lemma 4.2 of \cite{Bazlov}. 
For convenience, we recall the proof. 

\begin{proposition}     \label{prop2}

For any $p \in S\h$, one has $\hc(\tau(p)) = \ev_\rho(p)$. 
In particular, $\hc \circ \tau$ maps $S\h \subset U\g$ to $\C$.

\end{proposition}

\begin{proof}

%Let $\{ h_\beta, e_\alpha \ ; \ \beta \in \Pi, \, \alpha \in \Delta\}$ 
%be a Chevalley basis of $\g$.  
For $x \in \h$, we have
\[
	\tau(x) = \frac{1}{4} \, \sum_{\alpha \in \Delta_+} 
		\, \theta(e_\alpha) \theta([e_{-\alpha} , x]_\g)  
		+ \theta(e_{-\alpha}) \theta ([e_{\alpha} , x]_\g) 
	=  \frac{1}{4} \, \sum_{\alpha \in \Delta_+} 
		\, \alpha(x) (2 -2 e_{-\alpha} e_{\alpha}).
\]
As a result, 
\[
	\hc(\tau(x))= \frac{1}{2} \, \sum_{\alpha \in \Delta_+} \, \alpha(x) = \rho(x).
\] 
Since $\hc$ is an algebra homomorpism on $\ad_\h$-invariant elements (\cite[Lemma 2.4]{Bazlov}),
it follows that $\hc(\tau(p)) = \ev_\rho(p)$ for any $p \in S\h$, as required.

\end{proof}

Define two maps $\mu_{i}: (U\g \otimes \g)^\g \to \Cl(\h)$, for $i=1,2$, as follows. 
The first map,
\[�
	\mu_1: (U\g \otimes \g)^\g \to S\h \otimes \h \to \h \to \Cl(\h), 
\]
is the composition $\theta \circ (\ev_\rho \otimes 1) \circ \HC_\g$.
The second map, 
\[
	\mu_2 : (U\g \otimes \g)^\g \to \Cl(\g)^\g \to \Cl(\h),
\]
is the composition of $\mu : U\g \otimes \g \to \Cl(\g)$ 
and of the odd Harish-Chandra projection $\hc : \Cl(\g)^\g \to \Cl(\h)$.

\begin{proposition}    \label{prop3}

The maps $\mu_1$ and $\mu_2$ are equal to each other. 

\end{proposition}

\begin{proof}

By Proposition \ref{prop1},
$\mu_2(\alpha)=\mu_2 (\HC_\g(\alpha))$ for $\alpha \in (U\g \otimes \g)^\g$
(here we view $S\h \otimes \h$ as a subspace of $U\g \otimes \g$).
Furthermore, writing $\HC_\g(\alpha) = \sum_k a_k \otimes x_k$
with $a_k \in S\h, x_k \in \h$, one has 
\[
	\mu_2(\alpha)
	 	=  \sum_k \mu_2( a_k \otimes x_k) \\ 
	 	=  \sum_k \hc \circ \m \circ (\tau \otimes 1) (a_k \otimes x_k) \\
	 	=  \sum_k \hc(\tau(a_k) \theta(x_k)).  
\]
Recall that $\hc$ is an algebra homomorphism on $\ad_\h$-invariant
elements. 
Hence, by Proposition \ref{prop2}, we get: 
\[�
	\mu_2(\alpha)
		=  \sum_k \hc(\tau(a_k)) \hc(\theta(x_k)) \\
		=  \sum_k \ev_\rho(a_k) \theta(x_k) \\
		=  (\ev_\rho \otimes \theta) (\HC_\g (\alpha)) \\
		= \mu_1(\alpha).
\]

\end{proof}

Our last proposition is essentially due to Kostant \cite{Kostant_Cliff}. 

\begin{proposition}   \label{prop4}

For any $m \in \N$, we have 
$\mu\big( (U^{(m)}\g \otimes \g)^\g \big) = \J^{(2m+1)}.$

\end{proposition}

\begin{proof}

By Theorem D of \cite{Kostant_Cliff}, $\Cl(\g)$ is a free module over 
the subalgebra $\Cl(\g)^\g$. 
Namely, 
\[ 
	\Cl(\g) = \tau(U \g) \otimes \Cl(\g)^\g.
\]
Furthermore, $\tau$ maps the center $Z(U\g)$ 
of $U\g$ to $\C$ (see also \cite[Corollary 36]{Kostant_Cliff}). 
Corollary 80 and Theorem 89 of \cite{Kostant_Cliff} 
give an explicit expression for 
the elements in $\J^{(2 m_i +1)}$, $i=1,\ldots,r$: 
For any $p \in \J^{(2 m_i +1)}$, 
we have 
\[
	p �= \sum_{k} \tau (a_{k}) \theta(b_{k}), 
\] 
where $\sum_{k} a_{k} \otimes b_{k}$ is in $(U^{(m_i)}\g \otimes \g)^\g$. 

As a consequence, we obtain inclusions, 
$\J \subset  \mu\big( (U\g \otimes \g)^\g \big)$ 
and $\J^{(2 m_i +1)} \subset  \mu\big( (U^{(m_i)}\g \otimes \g)^\g \big)$. 
The vector space $\J$ has dimension $r$  
and $(U\g \otimes \g)^\g$ is a free module 
of rank $r$ over $Z(U\g)$.  
Then we deduce from the equality $\tau \big( Z(U\g) \big) = \C$
that $\mu \big( (U\g \otimes \g)^\g \big) = \J$. 
Moreover, since ${\rm deg}(\tau(x))=2$ 
and ${\rm deg}(\theta(x))=1$ for $x\in \g$, 
we get  
\[ 
	\mu \big( (U^{(m)}\g \otimes \g)^\g \big) \subset \J^{(2m+1)}.
\]
%\qA{I understand one inclusion '$\subset$' but give an argument 
%for the other one...}
Hence, $\mu \big( (U^{(m)}\g \otimes \g)^\g \big) = \J^{(2m+1)}$.

\end{proof}

We are now in the position to prove Theorem \ref{Kostant}:

\begin{proof}[Proof of Theorem \ref{Kostant}]

Joseph's theorem (Theorem \ref{Joseph}) states that 
\[�
	\mu_1\big( (U^{(m)}\g \otimes \g)^\g \big) = \theta(\c{\F}^{(m)}\h).
\]
We have to show that the above equality 
implies the Kostant conjecture (Conjecture \ref{Kostant}). 
In other words, 
we have to show that 
$\mu_1\big( (U^{(m)}\g \otimes \g)^\g \big)=  \hc( \J^{(2m+1)})$. 
But by Proposition \ref{prop3} and Proposition \ref{prop4}, 
we have 
\[
	\mu_1\big( (U^{(m)}\g \otimes \g)^\g \big) 
	= \mu_2 \big( (U^{(m)}\g \otimes \g)^\g \big) 
	=  \hc( \J^{(2m+1)}).
\]
%So, this concludes the proof of the theorem. 

\end{proof}

%%%%%%%%%%%%%%%%%
%%%
\section{More on Harish-Chandra projections}
%%%
%%%%%%%%%%%%%%%%%

The standard Harish-Chandra projection $\HC: U\g \to S\h$
is equivariant under the adjoint action of $\h$. Hence, the image
of $(U\g \otimes \g)^\g$ under $\HC \otimes 1$ lands in 
$S\h \otimes \h$.  
Our second result is the following equality of filtrations on $\h$:

\begin{theorem} \label{thm2}

\[�
	(\ev_\rho \circ \HC \otimes 1) (U^{(m)}\g \otimes \g) = \hc(\J^{(2m+1)}).
\]

\end{theorem}

We first state an auxiliary statement preparing the proof of this Theorem: 

\begin{proposition}  \label{prop:HC}

The image of the map $\HC -1: U\g \to U\g$ is contained in the kernel
of $\hc \circ \tau: U\g \to \Cl(\h)$.

\end{proposition}

\begin{proof} 

The image of the map $\HC -1$ is the space
$\n_- \, U\g + U\g \,  \n_+ \subset U\g$. 
For the first term, we have
\[ 
	\tau\big( \n_- \, U\g \big )= \tau(\n_-)\tau(U\g) \subset \theta(\n_-)\Cl(\g) \subset
	{\rm ker}(\hc).
\]
Indeed, as already observed, we have $\tau(\n_-) \subset \theta(\n_-)\Cl(\g)$ 
(see the proof of Proposition \ref{prop1}).  
Similarly, for the second term,
\[
	\tau\big( U\g \, \n_+ \big)=\tau(U\g) \tau(\n_+) \subset \Cl(\g)  \theta(\n_+) \subset
	{\rm ker}(\hc).
\]

\end{proof}

Now we are ready to present the proof of Theorem \ref{thm2}.

\begin{proof}[Proof of Theorem \ref{thm2}]

By Proposition \ref{prop:HC}, the following maps from $U\g$ to $\Cl(\h)$ coincide: 
\[
	\hc \circ \tau = \hc \circ \tau \circ \HC. 
\]
In turn, by Proposition \ref{prop2}, 
$\hc \circ \tau= \ev_\rho$ on the image of the Harish-Chandra
map $\HC: U\g \to S\h$. 
Hence, we obtain an equality of maps: 
\[
	\hc \circ \tau= \ev_\rho \circ \HC.
\]
This observation shows that the image of both maps is in fact equal to $\C \subset \Cl(\h)$.

For %$x\in \h$ 
$x \in \g$ define a linear form, $\beta_x: \g \to \C$, 
given by $\beta_x(y) = B_\g(x,y)$.
Then, for all $x \in \h$ and $\alpha \in (U\g \otimes \g)^\g$, we have
\[
	(\hc \circ \tau \otimes \beta_x)(\alpha)=(\ev_\rho \circ \HC \otimes \beta_x)(\alpha).
\]
For $x \in \h, v \in \J$, consider the expression
\[
	[\theta(x), \hc(v)] = \hc([\theta(x), v]).
\]
By Theorem \ref{BK}, $\hc(v) \in \theta(\h)$ and the above expression is 
equal to $B_\g (x , y)$, where $\hc(v) = \theta(y)$.  
Assume that it vanishes for all $v \in \J^{(2m+1)}$. 
By definition, it is equivalent to $x \in \hc(\J^{(2m+1)})^\perp$.

%The subspace $\tau(U\g)$ of $\Cl(\g)$ has a structure of $\g$-module. 
%Since $\g$ is simple, the adjoint representation is irreducible.  
%Let $\tau(U\g)_{\ad}$ be the  
%primary adjoint representation component in $\tau(U \g)$. 
%Then $\tau(U\g)_{\ad}$ is the subspace 
%of $\tau(U \g)$ generated by 
%$(\tau \otimes \beta_x) (U\g \otimes \g)^\g$ for $x \in \g$ 
%and Theorem E in \cite{Kostant_Cliff} 
%tells us that $\tau(U\g)_{\ad} = [\theta(x), \J ]$. 
%More precisely, 
% \qA{i don't from Theorem E why there is a surjection for any given $x \in \h$...}
By Theorem F (equality (n)) and Corollary 80 (equality (302)) in \cite{Kostant_Cliff}, 
%for any $q \in \J^{(2m+1)}$, 
%there is $u \in (U\g \otimes \g)^\g$ 
%such that 
%\[ 
%	[ \theta(x), q] = (\tau \otimes \beta_x)  (u)
%\]
for  any $v \in \J^{(2m+1)}$ there is an element $\alpha \in (U^{(m)}\g \otimes \g)^\g$
such that $[\theta(x), v]=(\tau \otimes \beta_x)(\alpha)$. 
Furthermore, the space $(U^{(m)} \otimes \g)^\g$ 
surjects on $[\theta(x), \J^{(2m+1)}]$ under the map 
$\tau \otimes \beta_x : U\g \otimes \g \to \Cl(\g)$ for $m \in \N$.  
Then,
\[
	\hc([\theta(x), v])=(\hc \circ \tau \otimes \beta_x)(\alpha)=
	(\ev_\rho \circ \HC \otimes \beta_x)(\alpha)=
	B(x, (\ev_\rho \circ \HC \otimes 1)(\alpha) ).
\]
Vanishing of this expression for any $\alpha \in (U^{(m)}\g \otimes \g)^\g$ 
is equivalent to 
$x \in \big((\ev_\rho \circ \HC \otimes 1)(U^{(m)}\g \otimes \g)^\g\big)^\perp$.
%\qA{Why this is equivalent to the fact that the expression 
%vanishes for any $v \in \J^{(2m+1)}$? }

\medskip

We have shown that the orthogonal complements of the two filtrations coincide.
Hence, so do the filtrations  in question.

\end{proof}

\begin{remark} 
While this paper was in preparation, Joseph informed us that he proved  \cite{Joseph2}
(using Zhelobenko invariants) the following equality of filtrations: 
\[ (\ev_\rho \circ \HC \otimes 1) (U^{(m)}\g \otimes \g) = \c{\F}^{(m)} \h.  
\]
Together with Theorem \ref{thm2}, this gives an alternative proof of the
Kostant Conjecture.

Joseph's result resembles of an earlier result by Rohr  \cite{Rohr} 
who showed that the filtrations $\c{\F}^{(m)} \h$ coincides with
the filtration $(\ev_\rho \circ \Phi \otimes 1)((S^{(m)}\g \otimes \g)^\g)$,   
where $\Phi$ is the projection map onto $S\h$ with respect to the decomposition, 
\[ 
	S\g = S\h \oplus S\g \big( \n_-  \oplus \n_+ \big).
\] 
In fact, Rohr's result is an essential step in the proof of Joseph's theorem.

\end{remark}


\begin{thebibliography}{AD04}
%\begin{thebibliography}{99}

\bibitem[AM00]{AM_Weil}
A. Alekseev and E. Meinrenken,
{\it The non-commutative Weil algebra},
Invent. math. {\bf 139} (2000), no. 1, 135-172.

\bibitem[AMW00]{AMW}
A. Alekseev, E. Meinrenken and C. Woodward,
{\it Group-valued equivariant localization},
Invent. math. {\bf 140} (2000), no. 2, 327-350.

\bibitem[AM11]{AM}
A. Alekseev and P. Mnev,
{\it One-dimensional Chern-Simons theory},
Commun. Math. Phys. {\bf 307} (2011), 185-227.

\bibitem[B]{Bazlov_thesis}
Y. Bazlov, 
{\it Exterior powers of the adjoint representation of a simple Lie algebra},
PhD Thesis, Weizmann Institute, 2003.

\bibitem[B2]{Bazlov}
Y. Bazlov,
{\it The Harish-Chandra isomorphism for Clifford algebras},
preprint arXiv:0812.2059.

\bibitem[J]{Joseph}
A. Joseph,
{\it Zhelobenko invariants, Bernstein-Gelfand-Gelfand operators
and the analogue Kostant Clifford algebra conjecture},
preprint arXiv:1109.5854.

\bibitem[J2]{Joseph2}
A. Joseph,
{\it Analogue Zhelobenko invariants for the Kostant and Hitchin
Clifford algebra conjectures}, 
in preparation.

\bibitem[KNV11]{KNV}
S. Khoroshkin, M. Nazarov and E. Vinberg,
{\it A generalized Harish-Chandra isomorphism},
Adv. Math. {\bf 226} (2011), 1168-1180.

\bibitem[KO08]{KO}  
S.Khoroshkin and O.Ogievetsky, 
{\it Mickelsson algebras and Zhelobenko operators}, 
J. Algebra {\bf 319} (2008), 2113-2165.

\bibitem[K59]{Kostant_sl2}
B. Kostant,
{\it The principal three-dimensional subgroup and the Betti numbers
of a compex simple Lie group}, 
Amer. J. Math. {\bf 81} (1959), 973-1032.

\bibitem[K97]{Kostant_Cliff}
B. Kostant, 
{\it Clifford Algebra Analogue of the Hopf-Koszul-Samelson Theorem,
the $\rho$-Decomposition on $C(\g)={\rm End} \, V_\rho \otimes C(P)$,
and the $\g$-module Structure on $\wedge \g$},
Adv. Math. {\bf 125} (1997), 275-350.

\bibitem[K03]{Kostant_Dirac}
B. Kostant,
{\it Dirac Cohomology for the Cubic Dirac Operator},
Studies in memory of Issai Schur (Chevaleret/Rehovot, 2000), 
69-93, 
Progr. Math., {\bf 210}, Birkh\"{a}user Boston, 2003. 



\bibitem[R10]{Rohr} R. Rohr, 
{\it Principal basis in Cartan subalgebra}, 
J. Lie Theory {\bf 20} (2010), no. 4, 673-687.


\end{thebibliography}
\end{document}